\documentclass[11pt]{tran-l}
\usepackage{tikz}
\usetikzlibrary{patterns}

\usepackage[centertags]{amsmath}
\usepackage{amsfonts}
\usepackage{amssymb}
\usepackage{amsthm}
\usepackage{graphics, graphicx}
\usepackage{newlfont}
\usepackage{tabularx}

\theoremstyle{plain}
\newtheorem{thm}{Theorem}[section]
\newtheorem{cor}[thm]{Corollary}
\newtheorem{lem}[thm]{Lemma}
\newtheorem{prop}[thm]{Proposition}
\theoremstyle{definition}
\newtheorem{defn}{Definition}[section]
\theoremstyle{remark}

\theoremstyle{Example}

\theoremstyle{Open Question}

\theoremstyle{Conjecture}

\numberwithin{equation}{section}

\begin{document}
\title[
The Generalized Randic Index 
]
{
On Clique Version of the Randic Index
}
\author{Hossein Teimoori Faal}

\address
{
Department of Mathematics and Computer Science,\\
Allameh Tabataba’i University, Tehran, Iran
}

%\email{hossein.teimoori@gmail.com\\iran@kam.mff.cuni.cz}

\maketitle
%%% ----------------------------------------------------------------------

\begin{abstract}

In this paper, we first review the weighted-version 
of the handshaking lemma based on the idea of a weighted 
vertex-edge incidence matrix of a given graph $G$.
Then, we obtain a generalized version of the 
handshaking lemma based on the concept of the clique value. 
We also define a generalized version of Randic index. More importantly, we prove an upper bound for the generalized Randic index of a graph $G$. We finally concluse the paper with some disscussions about possible future works.

\end{abstract}

%%% ----------------------------------------------------------------------
\section{Introduction and Motivation}

One of the important 
parameters of a simple, finite graph is the \emph{degree} 
of a vertex. It simply reflects a topological property of a graph which is the size of an \emph{open neighborhood}
of a given vertex. Therefore, one interesting line of research in graph theory is to generalize this local concept and also seeking for its potential applications. 
\\
In \cite{Teimoori20}, the author of this paper has suggested an extension 
this concept to the value of an edge $e=\{u,v\}$
as the size of the intersection of the open neighborhoods of 
its end-vertices; that is 
$
val_{G}(e) = \vert N_{G}(u) \cap N_{G}(v) \vert 
$
.
He also has used this idea in \cite{Teimoori20} to find a new upper bound for the number of edges with respect to the number of triangles in any $K_{4}$-free graph. 
\\
The well-known $Randi\check{c}$ index $R(G)$ of a graph 
$
G
$
was introduced in 1975 by $Randi\check{c}$ \cite{Randic75}. More precisely, he has defined it by 

\begin{equation}
R(G) :=
\sum_{\{u,v\} \in E(G)}
\frac{1}{\sqrt{deg_{G}(u) deg_{G}(v)}}. 
\end{equation}

Indeed, in this paper we will call it the \emph{vertex-version} of $Randi\check{c}$ index. 
It seems that this parameter is very useful in 
\emph{mathemaical chemistry} and has been extensively investigated in the literature 
(see \cite{LiGutman06} and the references therein). 
It is worth to mention the following two classical results in the context of $Randi\check{c}$ index. 

\begin{thm}[Bollobas and Erdos \cite{BollobasErdos98}]
	
	For any connected graph $G$ with $n$ vertices, 
	$
	R(G) \geq \sqrt{n-1}
	$
	with equality if and only if $G \cong K_{1,n-1}$.

\end{thm}

\begin{thm}[Fajtlowicz \cite{Fajtlowicz22}]\label{keyR1}
	
	For a graph $G$ with $n$ vertices, 
	$
	R(G) \leq \frac{n}{2}
	$
	with equality if and only if each component of 
	$G$ has alt least two vertices and is regular. 
	
\end{thm}

Our main goal in this paper is to obtain an extension of 
Theorem \ref{keyR1}, based on the idea of 
the value of a clique.

\section{Basic Definitions and Notations}

We assume that our graphs are simple, finite and undirected. 
For a given graph $G=(V,E)$ and a vertex 
$v \in V(G)$, the set of vertices adjacent 
to $v$ is called the \emph{open neighborhood}
of $v$ in $G$ and will be denoted by 
$
N_{G}(v)
$
.

The cardinality of $N_{G}(v)$ is called the 
\emph{degree} of $v$ and is denoted by $\deg_{G}(v)$.
A complete subgraph of $G$ is called a 
\emph{clique} of $G$. A clique on $k$ vertices is 
called a $k$-clique. A clique on $3$ vertices is called 
a triangle. We will denote the set of triangles of 
$G$ by $T(G)$. We denote the set of all $k$-cliques 
of $G$ by 
$
\Delta_{k}(G)
$
.
The number of $k$-cliques of $G$ is denoted 
by $c_{k}(G)$.  
We also recall that the well-known 
\emph{Cauchy-Schwartz} is the following inequality.  

\begin{lem}\label{GHMIneq} 
	[Geometric-Harmonic Mean Inequality]
	For any real sequences $\{a_{k}\}_{k\geq 1}$,
	we have 
	
	\begin{equation}
	\sqrt[k]{ a_{1} a_{2} \cdots a_{k} }
	\geq 
	\frac{k}
	{
		\frac{1}{a_{1}} + \frac{1}{a_{2}} + \cdots + \frac{1}{a_{k}}
	}
	,
	\end{equation}	
	with equality whenever $a_{1} = a_{2} = \cdots = a_{k}$. 
	
\end{lem}

In what follows, we quickly review the 
weighted-version of the well-known handshaking lemma 
and one of its consequences which is known 
as Mantel's theorem for triangle-free graphs. 
\\
The \emph{weighted-version} of the well-known 
\emph{handshaking lemma} can be read, as follows. 
From now on, we will denote the set of non-negative 
real numbers with $\mathbb{R}^{+}$. 

\begin{lem}
	[Weighted Handshaking Lemma \cite{Wu14}]
	Let $G=(V,E)$ be a graph and 
	$f:~V(G) \mapsto \mathbb{R}^{+}$ be a non-negative 
	\emph{weight} function. Then, we have 
	
	\begin{equation}\label{verthand1}
	\sum_{v \in V(G)} f(v) \deg_{G}(v) =
	\sum_{e=uv \in E(G)}
	\Big(
	f(u) + f(v)
	\Big).
	\end{equation}
	
	In particular, we have
	
	\begin{equation}\label{KeyIdent1}
	\sum_{v \in V(G)} \deg^{2}_{G}(v) =
	\sum_{e=uv \in E(G)}
	\Big(
	\deg_{G}(u) + \deg_{G}(v)
	\Big).
	\end{equation}
	
\end{lem}

\section{A vertex-version of Randic index}

In chemical graph theory literature, the \emph{branching index} of a given graph
$
G
$
is known as the 
$Randi\check{c}$ index of the graph $G$ denoted by $R(G)$. Here, we will denote it 
by $R_{v}(G)$ and call it \emph{vertex-version} $Randi\check{c}$ index of $G$. It is 
defined, as follows 
 
\begin{equation}
R_{v}(G) = 
\sum_{e=uv \in E(G)} 
\frac{1}
{\sqrt{\deg_{G}(u) \deg_{G}(v)}  }
.
\end{equation} 

In \cite{Wu14}, the following result is proved by a simple argument based on the 
\emph{weighted version of handshaking} lemma and 
\emph{geometric-harmonic mean} inequality.

\begin{thm}[\cite{Fajtlowicz22}]\label{KeyR1}
	
	For a graph $G$ of order $n$, 
	\begin{equation*}
	R_{v}(G) \leq \frac{n}{2},
	\end{equation*}	
	with equality if and only if every component of $G$ is regular and $G$ has no isolated vertices. 
	
\end{thm}

For the sake of completeness, here we also provide a short proof of 
Theorem \ref{KeyR1} based on the reference \cite{Wu14}. 

\begin{proof}
	
First, we note that by defining the function $f$ in the equation (\ref{verthand1}) by 
	
	\[ 
	\big(
	f(v) 
	= 
	\begin{cases}
	\frac{1}{\deg_{G}(v)} & \textnormal{if}~ 
	\deg_{G}(v) > 0
	, \\
	1 & \textnormal{if}~ \deg_{G}(v) = 0, 
	\end{cases} 
	\]  	
	we obtain
	
\begin{equation}\label{keyIdent1}
	\sum_{uv \in E(G)}
	\Big(
	\frac{1}{\deg_{G}(u)} + \frac{1}{\deg_{G}(v)} 
	\Big)
	=
	n - n_{0},
\end{equation}
	
where $n_{0}$ denotes the number of isolated vertices in $G$. 
\\
Now considering the \emph{geometric-harmonic} inequality for $k=2$, Lemma \ref{GHMIneq}, we have 
	
\begin{eqnarray}
R_{v}(G)
& = &   
\sum_{e=uv \in E(G)} 
\frac{1}
{\sqrt{\deg_{G}(u) \deg_{G}(v)} } \\ \nonumber 
& \leq &   
	\sum_{uv \in E(G)}
	\frac{1}{2}
\big(
	\frac{1}{\deg_{G}(u)} + \frac{1}{\deg_{G}(v)} 
	\big) \\ \nonumber 
& = &
	\frac{n - n_{0}}{2}
	\leq 
	\frac{n}{2},
\end{eqnarray}
as required.

\end{proof}

\section{An edge-version of Randic index}

In this section, we aim to obtain an edge version  
of the upper bound for the $Randi\check{c}$ index, based on 
the new concept of the value of an edge.    

\begin{defn}
	
	Let $G=(V,E)$ be a graph and $e=uv$ be an 
	edge of $G$. Then, we define the \emph{edge value}
	of $e$, denoted by $val_{G}(e)$, as follows
	
	\begin{equation}
	val_{G}(e) = \vert N_{G}(e) \vert = \vert N_{G}(u) \cap N_{G}(v) \vert. 
	\end{equation}		
	Here, $N_{G}(e)$ denotes the set of common neighbors 
	of end vertices of the edge $e$. 		
	
\end{defn}

Next, we generalize the weighted handshaking lemma 
for values of edges of a given graph. 

\begin{lem}
	
	[Weighted Edge Handshaking Lemma]
	Let $G=(V,E)$ be a graph and 
	$g:~E(G) \mapsto \mathbb{R}^{+}$ be a non-negative 
	\emph{weight} function. Then, we have 
	
	\begin{equation}\label{edgehand1}
	\sum_{e \in E(G)} g(e) val_{G}(e) =
	\sum_{\delta=e_{1}e_{2}e_{3} \in T(G)}
	\Big(
	g(e_{1}) + g(e_{2}) + g(e_{3}) 
	\Big).
	\end{equation}
	
	In particular, we have
	
	\begin{equation}\label{KeyIdent2}
	\sum_{e \in E(G)} val^{2}_{G}(e) =
	\sum_{\delta=e_{1}e_{2}e_{3} \in T(G)}
	\Big(
	val_{G}(e_{1}) + val_{G}(e_{2})+ val_{G}(e_{3})
	\Big).
	\end{equation}

\end{lem}

As an immediate consequence of the above lemma, we have the following interesting result. Recall that an edge $e \in E(G)$
is called \emph{isolated} if we have 
$
val_{G}(e) =0
$
.

\begin{cor}\label{keycor1}

	For any graph $G$ with $m$ edges, we have 
	
	\begin{equation}
	\sum_{\delta=e_{1} e_{2} e_{3} \in T(G)}
	\bigg(
	\frac{1}{val_{G}(e_{1})} + \frac{1}{val_{G}(e_{2})} + 
	\frac{1}{val_{G}(e_{3})}
	\bigg) =
	m - m_{0},
	\end{equation}	
	in which $m_{0}$ is the number of \emph{isolated edges}.

\end{cor}

Nect, we give a generalization of 
the result in \cite{Fajtlowicz22}. To do so, we first need to give a 
generalization of the concept of $Randi\check{c}$ index. 

\begin{defn}
	
	For a given graph $G=(V,E)$, the edge-version of 
	the $Randi\check{c}$ index, denoted by $R_{e}(G)$, is defined as 
	
	\begin{equation}
	R_{e}(G) :=
	\sum_{\delta=e_{1} e_{2} e_{3} \in T(G)} 
	\frac{1}
	{\sqrt{val_{G}(e_{1}) val_{G}(e_{2}) val_{G}(e_{3})  }  }
	\end{equation}
	
\end{defn}

From now on, we will call a graph \emph{edge-regular}
if $val_{G}(e)$ is the same for all edges of $G$.

\begin{thm}\label{KeyR2}
	
	For a graph $G$ of size $m$, we have 
	
	\begin{equation*}
	R_{e}(G) \leq \frac{m}{3},
	\end{equation*}	
	
with equality if and only if every component of $G$ is edge-regular and $G$ has no isolated edges. 
	
\end{thm}

\begin{proof}

Considering Corollary \ref{keycor1} and arithmetic-geometric inequality, we have 	
	\begin{eqnarray}
	R_{e}(G) &:= &
	\sum_{\delta=e_{1} e_{2} e_{3} \in T(G)} 
	\frac{1}
	{\sqrt{val_{G}(e_{1}) val_{G}(e_{2}) val_{G}(e_{3})  } }
	\\ \nonumber 
	& \leq &
	\frac{1}{3}
	\sum_{\delta=e_{1} e_{2} e_{3} \in T(G)}
	\bigg(
	\frac{1}{val_{G}(e_{1})} + \frac{1}{val_{G}(e_{2})} + 
	\frac{1}{val_{G}(e_{3})}
	\bigg) \\ \nonumber 
	& = &
	\frac{1}{3}(m - m_{0})
	\leq 
	\frac{m}{3}.  
	\end{eqnarray}	
	
\end{proof}

\section{the clique handshaking lemma}

In this section, we attempt to find 
a more generalized version of $Randi\check{c}$ index which we call it \emph{the generalized}
$Randi\check{c}$ index. 
To this end, we first need to present 
an extension of the definition of the degree 
of vertex to the value of any clique of higher order.

\begin{defn}

Let 
$G=(V,E)$ be a graph and $q_{k} \in \Delta_{k}(G)$
	be a $k$-clique in $G$. Then, we define the value 
	of the clique 
	$q_{k}$ with the vertex set 
$V(q_{k})=v_{i_1}\cdots v_{i_k}$
denoted by $val_{G}(q_{k})$, as follows
	
	\begin{equation}
	val_{G}(q) = \bigg\vert \bigcap_{i=1}^{k} N_{G}(v_{j_{i}}) \bigg\vert . 
	\end{equation}
We will also call an $k$-clique $q_{k}$ with $val_{G}(q_{k}) = 0$, an \emph{isolated}
clique of $G$.	
\end{defn}

Note that any $k$-clique 
$q_{k}=v_{i_1}\cdots v_{i_k}\in \Delta_{k}(G)$
in $G$ can also 
be represented (uniquely) by 
$
q_{k}=q_{k-1,1}\cdots q_{k-1,k}
$
where for each $i=1,\ldots,k$ the symbol $q_{k-1,i}$
denotes a $(k-1)$-clique subgraph of $q_{k}$. We will use this 
fact in our next key lemma.

\begin{lem}
	[Weighted Clique Handshaking Lemma]
	Let $G=(V,E)$ be a graph and and let  
	$h:~\Delta_{k}(G) \mapsto \mathbb{R}^{+}$ 
	$(k\geq 2)$ 
	be a \emph{weight} function. Then, we have 
	
	\begin{equation}\label{cliquhand1}
	\sum_{q_{k} \in \Delta_{k}(G)} h(q_{k}) 
	val_{G}(q_{k}) =
	\sum_{q_{k+1}=q_{k,1}\cdots q_{k,k+1} \in 
		\Delta_{k+1}(G)}
	\Big(
	h(q_{k,1}) + \ldots + h(q_{k,k+1}) 
	\Big).
	\end{equation}
	
	In particular, we have
	
	\begin{equation}\label{KeyIdent3}
	\sum_{q_{k} \in \Delta_{k}(G)} val^{2}_{G}(q_{k}) =
	\sum_{q_{k+1}=q_{k,1}\cdots q_{k,k+1} 
		\in \Delta_{k+1}(G)}
	\Big(
	val_{G}(q_{k,1}) + \cdots + val_{G}(q_{k,k+1})
	\Big).
	\end{equation}

\end{lem}

\begin{proof}
	
	We proceed by defining the weighted 
	\emph{subclique-superclique}
	matrix 
	$I_{f,k}(G)$
	of order $k$, as follows

	\[ 
	\big(
	I_{f,k}(G)
	\big)_{q_{k},q_{k+1}} = 
	\begin{cases}
	h(q_{k}) & \textnormal{if}~ 
	q_{k}~\textnormal{is~a~subgraph~of}~
	q_{k+1}
	, \\
	0 & \textnormal{
		otherwise	
	}.
	\end{cases} 
	\]  	
	
	Next, we note that in the matrix $I_{f,k}(G)$
	each row corresponding to the clique $q_{k}$	
	has $val_{G}(q_{k})$ non-zero entries. Hence, 
	the resulting row-sum equals to 
	$
	h(q_{k})val_{G}(q_{k})
	$ 	
	.
	On the other hand, each column corresponding to 
	the clique $q_{k+1}=q_{k,1}\cdots q_{k,k+1}$ has
	the column-sum 
	$
	h(q_{k,1}) + \cdots + h(q_{k,k+1})
	$ 
	.
	Thus, by summing over all rows and columns 
	and equating them we get the desired result.

\end{proof}

\section{The Generalized Randic Index}

In this last section, we obtain a more generalized version of $Randi\check{c}$ index
which we call it the \emph{generalized $Randi\check{c}$ index}.  
\\
Next, we define the \emph{generalized $Randi\check{c}$ index} of a 
graph $G$ or a \emph{cliuqe-version} of the $Randi\check{c}$ index based on 
the new concept of the \emph{clique value} in graph theory. 

\begin{defn}
	
	Let $G=(V,E)$ be a graph. Then, the \emph{generalized $Randi\check{c}$ index} of $G$, denoted by $R_{cliq}(G)$, is defined by 
	
	\begin{equation}
	R_{clq}(G;k) :=
	\sum_{q_{k+1}=q_{k,1}\cdots q_{k,k+1} \in 
		\Delta_{k+1}(G)}
	\frac{1}
	{
		\sqrt{	
			\prod_{j=1}^{k+1} val_{G}(q_{k,j})
		}
	}, ~~~(k \geq 1)
	\end{equation}
	
\end{defn}

We first need the following key result. 

\begin{prop}\label{keyprop1}
	
	Let $G=(V,E)$ be a graph. Then, we have 
	
	\begin{equation}
	\sum_{q_{k+1}=q_{k,1}\cdots q_{k,k+1} \in 
		\Delta_{k+1}
	}	
	\bigg(
	\frac{1}{val_{G}(q_{k,1}) } + \cdots + \frac{1}{val_{G}(q_{k,k+1}) }
	\bigg)
	=
	c_{k}(G) - c_{k,0}(G),
	\end{equation}	
	in which $c_{k,0}(G)$ is the number of isolated 
	$k$-cliques of $G$. 	
\end{prop}

Now, we are at the position to state the main result of this paper which a generalized version of the result in \cite{Fajtlowicz22}. Form now one, a graph in which all 
values of $val_{G}(q_{k})$ are the same for $k$-cliques of $G$
is called $k$-\emph{clique regular} graph.

\begin{thm}
	Let $G$ be a graph. Then, we have 
	
	\begin{equation}
	R_{clq}(G;k) \leq \frac{c_{k}(G)}{k+1} . 
	\end{equation}
The equality holds if and only if each component of $G$ is a $k$-clique regular graphs. 	
\end{thm}

\begin{proof}

The proof is straight forward based on Proposition \ref{keyprop1} and 
the arithmetic-geometric mean inequality \ref{GHMIneq}. 
	
\end{proof}

\end{document}